\documentclass[12pt]{amsart}
\usepackage[T1]{fontenc}
\usepackage{amscd, amssymb,latexsym,amsmath, amscd, amsmath}
\usepackage{mathrsfs}
\usepackage{anysize}
\usepackage[sc]{mathpazo} 
\usepackage{hyperref}
\usepackage{caption}
\usepackage{mathtools}
\marginsize{2.5cm}{2.5cm}{2.5cm}{2.5cm}
\newtheorem{theorem}{Theorem}[section]
\newtheorem{lemma}[theorem]{Lemma}
\newtheorem{corollary}[theorem]{Corollary}
\newtheorem{proposition}[theorem]{Proposition}
\usepackage[
backend=biber,
style=alphabetic,
sorting=nyt
]{biblatex}

\bibliography{bibliography} 

\theoremstyle{definition}

\theoremstyle{remark}
\newtheorem{remark}{Remark}

\title{Atkin--Lehner Decompositions for Quaternionic modular forms}
\author{siddharth ramakrishnan cherukara }

\begin{document}

\maketitle
\section{\textbf{Introduction}}
Let $S_{k}(N)$ denote the space of weight $k$, level $N$ cusp forms. 
 The subspace of $S_{k}(N)$ spanned by modular forms coming from $S_{k}(M)$ for all $M\mid N$, $M\neq N$ is denoted by $S_{k}^{old}(N).$ Its elements are called \emph{oldforms of level $N$}. The subspace orthogonal to $S_{k}^{old}(N)$ with respect to the \emph{Petersson inner product} is called the \emph{newspace of level $N$}, and denoted by $S_{k}^{new}(N).$ Atkin--Lehner in \cite{MR268123} and Li in \cite{MR369263} showed that  \begin{enumerate}
     
     \item $S_{k}^{new}(N)($resp.$S^{old}_{k}(N))$ has an eigenbasis  for $T_{m}$, for all $m$ $($resp.$(m,N)=1)$, where $\{T_{m}\}_{m\geq1}$ denote the family of Hecke operators.
     \item Given an $N$,  how $S^{old}_{k}(N)$ decomposes as a direct sum of subspaces generaed by $S^{new}_{k}(M)$ for $M\mid N$ and  $M\neq N$. 
    
 \end{enumerate} The authors in \cite{MR1102225} develop a similar result for Hilbert modular forms of a fixed level. By the work of Jacquet--Langlands, these classical spaces correspond to spaces of cuspidal automorphic representations of $GL(2)$ over adeles. Using this interpretation, Casselman in \cite{MR337789} described how the results in \cite{MR268123} can be viewed as a purely representation-theoretic result. 
 
 Let $F_{\mathfrak{p}}$ be a local field and $\mathfrak{o}_{F}$ be its ring of integers. Let $\pi_{\mathfrak{p}}$ be an irreducible, admissible representation of $GL(2,F_{\mathfrak{p}})$ with trivial central character. Casselman considered the following family of compact open subgroups $$\Gamma_{0}(\mathfrak{b})=\{\begin{bmatrix}
a & b \\
c & d
\end{bmatrix}:c\equiv 0 (\bmod \mathfrak{b})\} \subset GL(2,\mathfrak{o}_{F})$$for any ideal $\mathfrak{b}\subset\mathfrak{o}_{F}$.  He studied $\pi^{\Gamma_{0}(\mathfrak{b})}$ and proved that, in fact, there exists some $\mathfrak{b}$ such that this space is non-zero and, for the largest such ideal, this space is in fact one-dimensional. We denote this ideal by $\mathfrak{p}^{c(\pi)}$ and call $c(\pi)$ \emph{the conductor of $\pi$}. Casselman also computed the dimensions of $\mathfrak{p}^m$, where $m>c(\pi)$. Using these results, Casselman reinterpreted results in \cite{MR268123}. 

We aim to do an analogous process in the case of quaternionic modular forms, continuing the work in \cite{MR4127854}. Let $F$ be a totally real field, $\mathfrak{o}$ its ring of integers, and $B$ be a definite quaternion division algebra over $F$ of discriminant $\mathfrak{D}_{B}$. Let $\mathcal{O}$ be an order in $B$. We can consider the space $M_{\textbf{k}}(\mathcal{O})$ of quaternionic modular forms of level $\mathcal{O}$ and weight $\textbf{k}$,  whose technical definition we defer to later. Analogous to the case of classical modular forms, they are automorphic forms on $B^{\times}(\mathbb{A})$, where $\mathcal{O}$ plays a similar role to the congruence subgroup of level $N$ in the classical case. For each ideal $\mathfrak{n}\subset\mathfrak{o}$, we can consider a double coset operator $T_{\mathfrak{n}}(\mathcal{O})$ on $M_{\textbf{k}}(\mathcal{O})$, which play the role of the Hecke operators similar to the case of classical modular forms. These are given by \emph{Brandt matrices}, introduced by Brandt in \cite{Brandt1943}. We refer the reader to \cite{MR4127854} for a detailed discussion on them. Following \emph{op.cit.} we denote by $\mathcal{H}^{S}(\mathcal{O})$  the unramified Hecke algebra.

Let $\mathfrak{p}$ be a finite prime of $F$. Similar to the work of Casselman, we aim to use representation theoretic results  of $B_{\mathfrak{p}}^{\times}$, to study the decomposition of $M_{\textbf{k}}(\mathcal{O})$. When $\mathfrak{p}$ does not divide $\mathfrak{D}_{B}$, $B_{\mathfrak{p}}^{\times}\cong GL(2,F_{\mathfrak{p}})$ and hence we can apply the results of Casselman. Therefore we are interested in the case where $\mathfrak{p}\mid\mathfrak{D}_{B}$. A difference from the work of Casselman is that here we have to consider two types of families of compact-open subgroups of the multiplicative group of local division algebra, instead of one. We describe these families briefly here. Let $K_{\mathfrak{p}}$ be a quadratic extension of $F_{\mathfrak{p}}$, which is embedded into $B_{\mathfrak{p}}$.  The compact-open subgroups we consider are  $\mathfrak{o}_{K_{\mathfrak{p}}}^{\times}\mathcal{U}^n$, where  $\mathfrak{o}_{K_{\mathfrak{p}}}$ is the ring of integers of $K_{\mathfrak{p}}$. Let $\pi_{\mathfrak{p}}$ be a representation of $B^{\times}_{\mathfrak{p}}$. The restriction problem we are interested in is studying $\pi_{\mathfrak{p}}^{\mathfrak{o}^{\times}_{K_{\mathfrak{p}}}}$, which behaves differently depending on whether $K_{\mathfrak{p}}$ is ramified or not. Our aim is to address the cases which are unsolved in \cite{MR4127854}. 

Here we give a brief comparison with what has been already done in \cite{MR4127854}. Let $\pi_{\mathfrak{p}},\mathfrak{o}^{\times}_{K_{\mathfrak{p}}}$ be as above, and we wish to study $\pi_{\mathfrak{p}}^{\mathfrak{o}^{\times}_{K_{\mathfrak{p}}}}$. The author in \cite{MR4127854} calculated the dimensions of the invariant spaces, in all cases except when $\pi_{\mathfrak{p}}$ has  odd conductor and $K_{\mathfrak{p}}$ is  ramified. We calculate dimensions of that case under the assumption that $\mathfrak{p}$ is not dyadic. Tunnell in \cite{MR721997} and Saito in \cite{MR1199206} prove how these dimensions can be studied via the epsilon factors associated to these representations.  This completes the restriction problem in the case of odd primes and we obtain applications to theory of both quaternionic and classical modular forms, some of which are noted below. This allows us to completely describe the newspace of quaternionic modular forms of odd level and the whole space of quaternionic modular forms of even level.

\subsection{{\textbf{Main Results}}}

Here we state the main results of the paper. As we mentioned before, this paper is divided into two parts: the local results and the corresponding global results. 
Here, we state the local result in full generality, however we only state our global results for prime power level over $\mathbb{Q}$, in order to minimize notation. 

Let $F$ be a totally real field, and let $\mathfrak{p}$ be a prime in $F$,lying over an odd prime $p$ in $\mathbb{Z}$. Let $q$ denote the cardinality of the residue field of $F_{\mathfrak{p}}$. Let $K_{\mathfrak{p}}$, $L_{\mathfrak{p}}$ denote its two non-isomorphic ramified quadratic extensions. Let $B$ be a definite quaternionic division algebra over $F$ with discriminant $\mathfrak{D}_{B}$, such that
$\mathfrak{p}|\mathfrak{D}_{B}$. Let $\mathcal{O}_{B}$ denote the maximal order in $B$, and $\mathfrak{P}$ its maximal ideal.
In what follows, we denote by $E_{\mathfrak{p}}$ a ramified extension of $F_{\mathfrak{p}}$ (i.e isomorphic to one of $K_{\mathfrak{p}}$ or $L_{\mathfrak{p}}$. Let $\pi_{\mathfrak{p}}$ denote a representation of $B^{\times}_{\mathfrak{p}}$, with $\dim\pi_{\mathfrak{p}}>1$, $c(\pi_{\mathfrak{p}})$ is odd. Then $\pi_{\mathfrak{p}}$ is dihedrally induced from a ramified quadratic extension of $F_{\mathfrak{p}}$. We refer the reader to Section 2.1.2 for a basic discussion on the representation theory of $B^{\times}_{\mathfrak{p}}$ and \cite{MR2234120} for a detailed discussion.
\begin{theorem}
Let $\pi_{\mathfrak{p}}$ be an irreducible smooth representation of $B^{\times}_{\mathfrak{p}}$ of odd conductor and trivial central character, dihedrally induced from $K_{\mathfrak{p}}$.     \begin{enumerate}
    \item $\dim \pi_{\mathfrak{p}}^{\mathfrak{o}^{\times}_{E_{\mathfrak{p}}}}=2$  if\begin{enumerate}
        \item 
    $E_{\mathfrak{p}}\cong K_{\mathfrak{p}}$ and  $q\equiv 3 (\bmod4)$ 
    \item  $E_{\mathfrak{p}}\cong L_{\mathfrak{p}}$ and $q\equiv 1 (\bmod4)$ 
   \end{enumerate} 
\item $\dim \pi_{\mathfrak{p}}^{\mathfrak{o}^{\times}_{E_{\mathfrak{p}}}}=0$  if \begin{enumerate}
        \item 
    $E_{\mathfrak{p}}\cong K_{\mathfrak{p}}$ and $q\equiv 1 (\bmod4)$ 
    \item  $E_{\mathfrak{p}}\cong L_{\mathfrak{p}}$ and $ q\equiv 3 (\bmod4)$ 
   \end{enumerate} 
 \end{enumerate}
\end{theorem}
Now for the global results, let us look at the case where $F=\mathbb{Q}$, and hence we write $p$ instead of $\mathfrak{p}$. We will also specialize to the case where $\mathfrak{D}_{B}=p$.
 Let $\mathcal{O}$ be a \emph{Bass order} in $B$ of level $p^r$, which means that \[\mathcal{O}\otimes \mathbb{Q}_{{v}}= \mathcal{O}_{v}=\begin{cases}
     \mathfrak{o}_{G_{p}}+\mathfrak{P}^{r-1} & v=p,\\
     M_{2}(\mathbb{Z}_{v})&v\neq p,\infty,\\
     
 \end{cases} 
 \] where $G_{p}$ is a quadratic extension of $F_{p}$.  We refer the author to \cite[Chapter~24]{MR4279905} for a comprehensive discussion of such orders. In the rest of the section, we fix such a $K_{p}$, and write $\mathcal{O}={O_{r}}(K_{p})$.  Also, let \[   
{E_{{p}}} \cong
     \begin{cases}
      K_{p} &   p  \equiv 3 (\bmod 4), \\
    L_{p} &  p  \equiv 1 (\bmod 4). \\ 
     \end{cases}
\]
 Let $S_{k}(\mathcal{O})$ denote space of  weight $k$, level $\mathcal{O}$  quaternionic cusp forms and let $S_{k}(N)$ denote the space of weight $k$, level $N$ cusp forms. Let $S_{k}^{new}(\mathcal{O})$ be the subspace of $S_{k}(\mathcal{O})$, which are new at the level. Let $S_{k}^{new-sc}(N)$ (resp. $S_{k}^{new,E_{p}}(N)$) denote the subspaces of $S_{k}^{new}(N)$, spanned by eigenforms with local representation $\pi_{p}$ is supercuspidal (resp. supercuspidal and dihedrally induced from $E_{p}$).

 \begin{theorem}
        \label{thm:main_result}
    Suppose level$(\mathcal{O})=p^{2r+1}$ with $r\geq 1$. 
As $\mathcal{H}^{S}$-modules we have isomorphisms
\begin{equation}
    S_{{k}}^{new}(\mathcal{O})\cong 2S^{new,{E}_{p}}_{{k+2}}(p^{2r+1}),\\ 
\end{equation}
\begin{equation}
S_{k}(\mathcal{O})\cong S_{{k+2}}^{odd}(\mathcal{O}) \bigoplus S_{{k+2}}^{even}(\mathcal{O}),\\
\end{equation}
 \textit{where} 
\begin{equation}
 S_{{k+2}}^{even}(\mathcal{O}) \cong 
    2S^{new-sc}_{{k+2}}(p^{2r})\oplus 2S^{new-sc}_{{k+2}}(p^{2r-2}) \oplus....\oplus 2S^{new-sc}_{{k+2}}(p^{4
    })\oplus 2S^{new-sc}_{{k+2}}(p^{2}),
\end{equation}
\textit{and} 
\begin{equation}
S_{{k+2}}^{odd}(\mathcal{O}) \cong 
    2S^{new,{E}_{p}}_{{k+2}}(p^{2r+1})\oplus 2S^{new,{E}_{p}}_{{k+2}}(p^{2r-1}) \oplus.... \oplus 2S^{new,{E}_{p}}_{{k+2}}(p^{3
    })\oplus S^{new}_{{k+2}}(p).\\
\end{equation}
\end{theorem}
\begin{theorem}
Suppose level$(\mathcal{O})=p^{2r}$ with $r> 1$.
As $\mathcal{H}^{S}$-modules we have isomorphisms
\begin{equation}
S_{{k}}^{new}(\mathcal{O})\cong 2S^{new-sc}_{{k+2}}(p^{2r}),
\end{equation}
and
\begin{equation}
S_{{k}}(\mathcal{O})\cong S_{{k+2}}^{odd}(\mathcal{O}) \bigoplus S_{{k+2}}^{even}(\mathcal{O}) ,\end{equation}
   where
\begin{equation}
S_{{k+2}}^{even}(\mathcal{O}) \cong 
    2S^{new-sc}_{{k+2}}(p^{2r})\oplus 2S^{new-sc}_{{k+2}}(p^{2r-2}) \oplus....\oplus 2S^{new-sc}_{{k+2}}(p^{2}) \oplus S^{new-sp}_{{k+2}}(p^2),
\end{equation} and 
\begin{equation}
S_{{k+2}}^{odd}(\mathcal{O}) \cong 
    2S^{new,{E}_{p}}_{{k+2}}(p^{2r-1})\oplus 2S^{new,{E}_{p}}_{{k+2}}(p^{2r-3}) \oplus.... \oplus 2S^{new,{E}_{p}}_{{k+2}}(p^{3
    })\oplus S^{new}_{{k+2}}(p).
\end{equation}
\end{theorem}

\begin{remark}
    In the Theorem 1.3 it is also possible that the $\mathcal{O}_{p}=\mathcal{O}_{2r+1}(M_{p})$, where $M_{p}$ is the unique unramified quadratic extension of $\mathbb{Q}_{p}$. This case is covered in \cite{MR4127854}. Also, the case of $r=1$ in Theorem 1.3 is also completely covered in \cite{MR4127854}, where the author additionally considers modular forms with non-trivial nebentypus. In contrast, the decomposition of new space in Theorem 1.2, the description of $S^{odd}_{k+2}(\mathcal{O})$ in Theorem 1.2 and Theorem 1.3 (i.e $(1),(4)$ and $(8)$) are entirely new contributions.
\end{remark}
\subsection{Applications}Here we describe some of the potential applications of our result and motivations behind our work. Our results provide a more complete description of the space of quaternionic modular forms. While quaternionic modular forms are interesting objects in their own right, one of the main reasons they were studied was to understand corresponding spaces of classical modular forms. One famous instance of this idea is the \emph{Basis Problem}. It was originally conjectured by Hecke in \cite{MR3665} that all cusp forms on $\Gamma_{0}(p)$ are linear combinations of theta series attached to norm forms of the definite quaternion division algebra of discriminant $p$. A corrected version of this conjecture was proved by Eichler in \cite{MR80768} and \cite{MR485698}. In \cite{MR579066}, Pizer gives an explicit algorithm to compute the spaces of modular forms based on above theory.  We refer the reader to \cite{MR960090} for a historical overview of the basis problem. In \cite{MR4127854}, the author gives a representation theoretic solution to Eichler's basis problem. For orders of level $p^{2r+1}$,  these references considered the case where they come from the unramified quadratic extension of $\mathbb{Q}_{p}$. But here we are able to consider the situation where they possibly arise from the ramified quadratic extension of $\mathbb{Q}_{p}$( see \ref{thm:main_result}), which is new. Hence we are able to give a more complete solution to Eichler's basis problem. Additionally, quaternionic modular forms are helpful in doing computations on the space of Hilbert modular forms as illustrated in \cite{MR3184337}, and we hope our results would help with some of these computational aspects.

Another application of our result is that we can isolate spaces of elliptic modular forms satisfying certain natural conditions. For simplicity let us consider the case where $N=p^{2r+1}$, for a fixed odd prime $p$ and $r\geq1$, an integer. Let $S^{new,E_{p}}_{k}(p^{2r+1})$ denote the subspace of $S^{new}_{k}(p^{2r+1})$, with local representation $\pi$ coming from dihedrally induced from $E_{p}$, where $E_{p}$ is a quadratic extension of $\mathbb{Q}_{p}$. The extension where the representation is induced is an invariant of the orbit of the representation under the action of the absolute Galois group. We say local representations $\pi_{p},\pi_{p}^\prime$ of $PGL(2,\mathbb{Q}_{p})$ are of the same inertial type if their restriction of $\rho,\rho^\prime$ to $I_{\mathbb{Q}_{p}}$, the inertial subgroup of $W_{\mathbb{Q}_{p}}$ are isomorphic, where $\rho,\rho^\prime$ are associated to $\pi_{p},\pi_{p}^\prime$ by the Local Langlands Correspondence. Now the Galois orbit of a given inertial type of $\pi$ consists of all representations which are induced from the same quadratic extension as $\pi_{p}$ and are having the same conductor when $p\geq5$ and the representation has an odd conductor. Hence in the result on the RHS, we can isolate a subspace of modular forms with local representation at $p$ of with a fixed local Galois type, by describing it as an isomorphic image of a space of quaternionic modular forms of a fixed level $\mathcal{O}$. Now suitably choosing the level $\mathcal{O}$, we can study modular forms with local representation at $p$ of any given type. We refer the reader to \cite{MR4269428} and \cite{KM} for a more complete discussion on Galois orbits of local types.

 In \cite{MR4492517}, the author uses results in \cite{MR4127854} to study the certain central $L$-values of Hilbert modular forms. Also in \cite{MR3762695} and \cite{MR4127854}, the author obtains certain congruence relations for suitable Eisenstein series with cusp forms. We hope our results would help in relaxing some of the assumptions there.

\section*{\textbf{Acknowledgements}}
 I would like to thank my advisor, Kimball Martin, for suggesting this problem and for his invaluable guidance. I am also grateful to him for carefully going through an earlier version of the paper and helping to improve the exposition. I also thank Basudev Pattanayak for helpful discussions. 
\section{\textbf{Preliminaries}}
\subsection{Local Preliminaries}
Let $F$ be a $p$-adic field, with  valuation $v_{F}$, normalised such that $v_{F}(F^{\times})=\mathbb{Z}$. Let $\mathfrak{p}$ denote all elements in $v_F$ with positive valuation. A \emph{quaternion algebra} over $F$ is a central simple algebra over $F$ of dimension 4. By the Hasse-Noether theorem, there are only two possiblities up to isomorphism, that is, the matrix algebra $M_{2}(F)$ and the unique quaternion division algebra, which we denote by $B$. We can extend $v_{F}$ to a valuation  $v$ on $B$, via $v=v_{F}\circ N_{B/F}$. Our main object of study will be  $B^{\times}$. Let 
$$\mathcal{O}_{B}=\{x\in B:v(x) \geq 0 \},$$  
$$\mathfrak{P}=\{x \in B:v(x)>0\},$$
$$\mathcal{O}_{B}^{\times}=\{x \in B:v(x)=0\},\textit{ and}$$
$$\mathcal{U}^{n}=1+\mathfrak{P}^{n},n\geq1.$$
Then $\mathcal{O}_{B}$ is the maximal order in $B$. We also consider other orders, namely special orders which are of the form $$\mathcal{O}_{r}(E)=\mathfrak{o}_{E}+\mathfrak{P}^{r-1}$$ where $r\geq1$, and $E$ is a quadratic extension of $F$ inside $B$. The level of a special order is $\mathfrak{p}^r$ with $r$ minimal such that $\mathcal{O}=\mathcal{O}_{r}(E)$, for some $E$. Hence by our convention $\mathcal{O}_{B}$ has level $\mathfrak{p}$. The details of these orders can be found in \cite{MR977435}. We also have $\mathcal{O}^{\times}_{r}(E)=\mathfrak{o}_{E}^{\times}\mathcal{U}^{r-1}$. The groups $\mathcal{U}^{n}$ are compact, open and normal subgroups of $B^{\times}$.
We aim to study the irreducible, smooth, admissible representations $\pi$ of $B^{\times}$. \emph{Smooth} means that, there exist an integer $n>0$,  $\pi$ is trivial on $\mathcal{U}^{n}$. We call the smallest such $n$ the \emph{conductor} of the representation $\pi$ and denote it as $c(\pi)$. \emph{Admissibile} means that $\pi^{\mathcal{U}_{r}}$ is finite dimensional for all $r>0$. Since $\mathcal{U}^{r}\trianglelefteq B^{\times}$, we see all such representations are finite dimensional.

A representation $\pi$ is called minimal if for any character $\phi:F^{\times} \rightarrow \mathbb{C}^{\times}$, $c(\pi) \leq c(\pi \otimes(\phi\circ N_{B/F}))$. 
Minimal representations of $B^{\times}$ are classified as follows:
\subsubsection{\textbf{One dimensional representations}} Let $\chi:F^{\times} \rightarrow \mathbb{C}^{\times}$ be a character of $F^{\times}$. Now $\mu=\chi \circ N_{B/F}$ is a character of $B^{\times}$. All characters of $B^{\times}$ are of this form.
\subsubsection{\textbf{Higher dimensional representations}}
 Fix an additive character $\psi$ of $F$ of conductor 0.
Let $E$ be a quadratic extension of $F$, let $\mathfrak{o}_{E}$ be the ring of integers of $E$, $\mathfrak{p}_{E}$ be the maximal ideal of $\mathfrak{o}_{E}$. Let $\chi$ be a character $E^{\times}$. The maximal $r >0$ such that $\chi$ is trivial on $1+\mathfrak{p}_{E}^{r}$ is called the \emph{conductor} of $\chi$ and denoted by $f(\chi)$. Let $N_{E/F}, T_{E/F}$ be the norm and trace maps from $E$ to $F$. Let $\chi$ be a minimal character of $E^{\times}$, i.e $f(\chi)\leq f(\chi.(\phi \circ N_{E/F}))$ for all characters $\phi$ of $F^{\times}$. We can divide this further into two cases
\begin{enumerate}
    \item \emph{Odd conductor representations}: Let $c(\pi)=2n+1\geq3$ and let $E$ be a ramified quadratic extension of $F$. Let $J=E^{\times}\mathcal{U}^{n}$. Then $\pi\cong Ind_{J}^{B^{\times}}\Lambda $, where $\Lambda$ is a character of $J$.
    All representations with odd conductor is obtained this way.
    \item \emph{Even conductor representations}: Let $c(\pi)=2n\geq2$ and let $E$ be the unramified quadratic extension of $F$. Let $J=E^{\times}\mathcal{U}^{n-1}$. Then $\pi\cong Ind_{J}^{B^{\times}}\Lambda $, where $\Lambda$ is a character of $J$ if $n$ is odd and a $q$-dimensional representation of J if $n$ is even. All minimal representations with even conductor is obtained this way.
\end{enumerate}
If $\pi$ is non-minimal, then $\pi$ always has even conductor.

\subsection{Global Preliminaries}
We briefly introduce quaternionic modular forms here. For a more detailed exposition, the reader is referred to \cite{MR4127854}.
Let $F$ be a totally real field of degree $d$, and $\mathfrak{o}=\mathfrak{o}_{F}$ be the ring of integers. Let $Sym^{k}$ be the k-th symmetric power of standard representation of $GL_{2}(\mathbb{C})$. Let $B$ be a definite quaternion algebra over $F$ and
let $\mathfrak{D}_{B}$ be the discriminant of $B$. An order $\mathcal{O}$ in $B$ is a subring of $B$, which is a complete $\mathfrak{o}_{F}$-lattice. For each finite place $\mathfrak{p}$ of $\mathfrak{o}_{F}$, let $\mathcal{O}_{\mathfrak{p}}=\mathcal{O}\otimes \mathfrak{o}_{F_{\mathfrak{p}}}$. We say $\mathcal{O}$ is a \emph{special order} of level $\mathfrak{N}$, if each local order $\mathcal{O}_{\mathfrak{p}}$ is of level $\mathfrak{p}_{v}^{n_{v}}$, and $\mathfrak{N=\prod p}_{v}^{n_{v}}$.
Let $\nu_{1}, \nu_{2},......\nu_{d}$ be the real embeddings of $F$. Hence $B_{\nu_{i}} \cong \mathbb{H}$, for all $i$. For such infinite places fix an embedding of $B_{\nu_{i}}^{\times}$ into $GL_{2}(\mathbb{C})$ and for $k_i>0$, compose this embedding with $Sym^{k}\otimes det^{-k/2}$, and call it $\rho_{k_i,\nu_i}$. Let $\textbf{k}=(k_{1},......k_{d})$, we call $\rho_{\textbf{k}}=\bigotimes_{\nu_{i}}\rho_{k_i,{\nu_{i}}}$.
Let 
$\hat{B}^{\times}=\{(x_{v})_{v};v<\infty$; such that $x_{v} \in \mathcal{O}^{\times}_{v}$, for all except finitely many $v\} \subset \prod_{v}B^{\times}_{v}$. Clealy $B^{\times} \hookrightarrow \hat{B}^{\times}$ via the diagonal embedding. Also $\hat{\mathcal{O}}^{\times}= \prod_{v<\infty}\mathcal{O}^{\times}_{v}$. 
 Let $Cl(\mathcal{O})$ denote the set of invertible fractional right ideal classes of $\mathcal{O}$ modulo principal fractional ideals.
 We have 
\begin{equation}
Cl(\mathcal{O})\cong B^{\times}\backslash {\hat{B}}^{\times}/\hat{\mathcal{O}}^{\times}
\end{equation}
 Let $M_{\mathbf{k}}(\mathcal{O})$ denote the space of quaternionic modular forms of weight $\mathbf{k}=(k_{\nu_{1}},......k_{\nu_{d}})$, and level $\mathfrak{N}$, defined as follows
$$M_{\mathbf{k}}(\mathcal{O})=\{\varphi:{B^{\times}}\backslash B^{\times}(\mathbb{A})  \rightarrow\mathbf{V}_{\mathbf{k}}:\varphi(x\alpha g)=\rho_{\textbf{k}}(g^{-1})\varphi(x); \textit{ for all } g \in B_{\infty}^{\times},\alpha \in \hat{\mathcal{O}}\}$$
Whenever $\mathbf{k=(0,0,....,0)}$, using (3) this means 
$$M_{\mathbf{k}}(\mathcal{O})\cong\{\varphi:Cl(\mathcal{O) \rightarrow\mathbb{C}}\}$$ 
Now let us define the cuspforms of level $\mathcal{O}$, denoted by $S_{\mathbf{k}}(\mathcal{O})$. When $\mathbf{k\neq(0,0,0...,0)}$, set $S_{\mathbf{k}}(\mathcal{O})=M_{\mathbf{k}}(\mathcal{O})$. When  $\mathbf{k=(0,0,0,0...,0)}$, we define $E_{\mathbf{k}}(\mathcal{O})$ to be the linear span of $\varphi=\mu \circ N_{B/F}$, for a character $\mu$ of $F^{\times}\backslash \mathbb{A}_{F}^{\times}$, such that $\varphi$ is in $M_{\mathbf{k}}(\mathcal{O})$. We can define an inner product on $M_{\mathbf{k}}(\mathcal{O})$, as $$(\varphi,\varphi^{\prime})= \int_{\mathbb{A}_{F}^{\times}B^{\times}\backslash B^{\times}(\mathbb{A})} \varphi(x) \overline{\varphi^{\prime}(x)} \,dx $$
for a suitably chosen Haar measure $dx$. We define $S_{\mathbf{k}}(\mathcal{O})$ to be the subspace of $M_{\mathbf{k}}(\mathcal{O})$ to be the subspace orthogonal to $E_{\mathbf{k}}(\mathcal{O})$ under this inner product.

\section{\textbf{Local Results}}
Let $F$ be a non-archimedean local field of characteristic $0$, with residue field having characteristic $p$ and cardinality $q$, where $p$ is odd. Denote $K$ and $L$  the ramified quadratic extensions of $F$, and $M$ the unique unramified quadratic extension of $F$.  Let $B$ be the unique quaternion division algebra over $F$. Let $\pi$ be an irreducible smooth representation of $B^{\times}$ with  trivial central character. Let $E$ be a quadratic  extension of $F$, and with normalised valuation $v_{E}$ such that $v_{E}(E^{\times})=\mathbb{Z}$. We call  $\nu_{E}(x)=(-1)^{v_{E}(x)}$ the sign character of $E^{\times}$. We know that $E^{\times}$ embeds into $B^{\times}$. As remarked in the introduction, we want to calculate  $\pi^{\mathfrak{o}^{\times}_{E}}$. This can be understood by fully understanding $\pi_{|E^{\times}}$. Let $\pi^{\prime}$ be the corresponding discrete series representation of $GL(2,F)$ via the Jacquet--Langlands correspondence. We know that the answer to our previous question can be described completely by knowing $\pi^{\prime}_{|E^{\times}}$, since 
\begin{equation}    
\dim( Hom_{E^{\times}}(\pi,\chi))+\dim( Hom_{E^{\times}}(\pi^{\prime},\chi))=1 \end{equation} by Waldspurger's work in \cite{MR1103429}, for any character $\chi:E^{\times}\rightarrow\mathbb{C}^{\times}$.
Let $\sigma$ be the irreducible 2 dimensional Weil--Deligne representation of $W_{F}$ associated to $\pi^{\prime}$ via Local Langlands Correspondence. All such representations are dihedrally induced from a  quadratic extension of $F$. In \cite{MR4127854}, author has completely worked out  $\dim\pi^{E^\times}$, when $\pi$ has even conductor and where $\pi$ has odd conductor and $E$ is unramified. We summarize the results in those cases from \cite{MR4127854} in \ref{table:Table}.

We will prove the remaining case, i.e $\pi$ has odd conductor and $E$ is ramified. So in what is to follow, we will assume that $E$ is a ramified quadratic extension of $F$, and the representation $\pi$ has odd conductor, dihedrally induced from $K$. The corresponding Weil--Deligne representation also arises from the index two subgroup $W_{K}$ of $W_{F}$. To each representation of $W_{F}$ to $\sigma$, we can associate a number called the $\varepsilon(\sigma,\psi)$-factor to $\sigma$, where $\psi$ is an additive character of $F$. We refer the reader to \cite{MR721997} for a basic introduction to $\varepsilon$-factors. Determining how the epsilon factor behaves under twisting completely determines the restriction problem described above. We need the following elementary lemma about norms from ramified quadratic extension. 
\begin{lemma}
    Let $E$ be any ramified quadratic extension and $N:E^{\times} \rightarrow F^{\times}$ denote the norm map. Then $N(\mathfrak{o}_{E}^{\times})$ is the set of squares in $\mathfrak{o}_{F}^{\times}$.
\end{lemma}
    \begin{proof}
        Clearly, the set of squares in $\mathfrak{o}_{F}^{\times}$ is contained in the image of the norm map. Let $\chi:F^{\times} \rightarrow \mathbb{C}^{\times}$ be the character associated with $E$. This means that $ker(\chi)=N(E^{\times})$. Looking at $\chi$ as a character $\mathfrak{o}_{F}^{\times}$, we know it factors through the quotient map $pr:\mathfrak{o}_{F}^{\times} \rightarrow \mathbb{F}_{q}^{\times}$. Therefore $\chi(x)=1 \iff \chi(pr(x))=1$. Since $\mathbb{F}_{q}^{\times}$ has a unique quadratic character, this means $\chi(p(x))=1$ if and only if $p(x)$ is a square in $\mathbb{F}_{q}^{\times}$. By Hensel's lemma, this implies $x$ is a square in $\mathfrak{o}_{F}^{\times}$, since $p$ is odd.
    \end{proof}
    Let $\lambda_{q}$ be the unique quadratic character of $\mathbb{F}_{q}^{\times}$. Now we call pull it back to get a character of $\mathfrak{o}_{F}^{\times}$. Let'sa  fix a uniformizer of $\varpi_{F}$. Now to extend it to a quadratic character of $F^{\times}$, we need to see what value it takes on $\varpi_{F}$. So there are two such characters with $\lambda_{1}(\varpi_{F})=1$ and
    $\lambda_{2}(\varpi_{F})=-1.$ Clearly these two differ by the unique unramified quadratic character of $F^{\times}$, $\nu$. For the following, let $K=F(\sqrt{-\varpi_{F}})$ and $L=F(\sqrt{-u\varpi_{F}})$, where $u$ is a nonsquare unit in $F$. Now $K$ corresponds to $\lambda_{1}$ and $L$ corresponds to $\lambda_{2}$. We can also take the uniformisers of $K$ and $L$ to be $\varpi_{K}=\sqrt{-\varpi_{F}}$ and 
$\varpi_{L}=\sqrt{-u\varpi_{F}}$. This choice is so that norm of these uniformizers are $\varpi_{F}$ and $u\varpi_{F}$ respectively. 
\begin{proposition}\cite[(1.1.5)]{MR721997}
    Let $\sigma$ be a representation of W$_{F}$ and let $\chi$ be an unramified character of W$_{F}$.

\begin{equation}
    \varepsilon(\sigma \bigotimes \chi,\psi)=\chi(\varpi_{F})^{a(\sigma)+n(\psi)dim(\sigma)}\cdot \varepsilon(\sigma, \psi)
\end{equation}
\end{proposition}    
\begin{proposition}
    Let $\sigma$ be an irreducible two-dimensional representation of $W_{F}$, with trivial central character induced from a quasicharacter $\varkappa$ of a ramified quadratic extension K. Then 
    \[\ \epsilon_{p}= \frac{\varepsilon(\sigma \bigotimes w_{L/F}, \psi)}{\varepsilon(\sigma,\psi)} = \begin{cases} 
          1 &  cond(\varkappa)=1\\
          1 &  L=K\\
          -1 & otherwise \\
       \end{cases}
    \]where $\psi$ is an additive character of $F$ with conductor 0.
    
    \end{proposition}
    \begin{proof}
       The method of proof is following Theorem 3.2 in \cite{MR3056552}. We know 
       $$\ \epsilon_{p}= \frac{\varepsilon(\varkappa \bigotimes {\tilde{w}_{L/F},\tilde{\psi})}}{\varepsilon(\varkappa,\tilde{\psi})}$$
where ${\tilde{w}_{L/F}=w_{L/F} \circ N_{K/F}}$, where $N_{K/F}:K^{\times} \rightarrow F^{\times}$ denote the norm map and ${\tilde{\psi}_{L/F}=\psi \circ Tr_{K/F}}$, where $Tr_{K/F}$ is the trace map from $K$ to $F$ and $cond(\tilde{\psi})=d(K/F)=1$. Lemma 3.1 implies that $\tilde{w}_{L/F}$ is an unramified character of K$^{\times}$. Hence $\varepsilon_{p}=\tilde{w}_{L/F}(\varpi_{K})^{a(\varkappa)+1}$. Now $w_{L/F}(N_{K/F}(\varpi_{K}))^{a(\varkappa)+1}=w_{L/F}(\varpi_{F})^{a(\varkappa)+1}=(-1)^{a(\varkappa)+1}$ if $L$ is not isomorphic to $K$ and $1$ otherwise. Since $\sigma$ has trivial central character, we know $a(\varkappa)=1$ or even. Hence it follows. By working  with a different uniformiser, we get the statement for all quadratic extensions. 
    \end{proof}
    In what is to follow, let $\chi$ be either the sign or trivial character of $E^{\times}(E$ could be $L$ or $K$, as above). Let $m(\sigma, \chi)$ denote the multiplicity of $\chi$ in $\sigma$ when restricted to $E^{\times}$.
 By the work of Tunnell in \cite{MR721997}, 
\begin{eqnarray}
m(\sigma,\chi)=(1+\varepsilon(\sigma_{E}\bigotimes\chi^{-1}))Det(\sigma(-1))/2   \\ 
\varepsilon(\sigma_{E}\bigotimes\chi^{-1})=\varepsilon(\sigma \bigotimes Ind(\chi^{-1}))w_{E/K}(-1)
\end{eqnarray}

 Hence we need to know what $Ind(\chi^{-1})$ is. (In our case $\chi=\chi^{-1})$. 
\begin{lemma}
     
\begin{equation}
        Ind_{E}^{F}(\chi)=(\chi \circ N_{E/F}) \bigoplus (\chi \circ N_{E/F}.w_{E/F})
    \end{equation}
\end{lemma}
\begin{proof}
    This follows from Frobenius reciprocity.
\end{proof}
Therefore it is enough to know the ratio $\frac{\varepsilon(\sigma \bigotimes w_{E/F})}{\varepsilon(\sigma)} $, for each quadratic extension $E$ of $F$. 
\begin{corollary}
    
    Let $\sigma$ be an irreducible two-dimensional representation of $W_{F}$, with trivial central character induced from a quasicharacter $\varkappa$ of a ramified quadratic extension K. Then trivial/sign character of $L^{\times}$ appears in the restriction of $\sigma$ to $L^{\times}$ if and only if $q\equiv 3\bmod 4$. 
\end{corollary}

\begin{proof}

    It is clear that since $\epsilon_{p}$  is $-1$, the multiplicity is $\frac{1-w_{L/F}(-1)}{2}$ and hence the result follows.
    \end{proof}
\begin{corollary}
    Let $\sigma$ be as above. The trivial or sign character  appears in the restriction of $\sigma$ to $K^{\times}$ if and only if q$\equiv 1\bmod 4$. 
\end{corollary}
\begin{proof}
    This is similar to the previous case.
\end{proof}

Let us summarize the results of this section. Let $\pi$ be a representation of $B^{\times}$ dihedrally induced from a ramified quadratic extension $K$ of $F$. Let $E$ be a ramified quadratic extension of F. Let $\chi$ be the sign or the trivial character of $E^{\times}$. 
\begin{theorem}
  
    \begin{itemize}
        \item      If $E\cong$K, then $m(\pi, \chi)=1$ if and only if q$\equiv 3 \bmod 4$.
        \item  If E is not isomorphic to K, then $m(\pi, \chi)=1$ if and only if q$\equiv 1 \bmod 4$.       
    \end{itemize}
\end{theorem}
\begin{proof}
    We know that a character $\chi$ of the torus $E^{\times}$ appears in the restriction of $\pi$ if and only if $\chi$ does not appear in the restriction of the Jacquet--Langlands transfer of $\pi$ to $E^{\times}$, this follows.
\end{proof}
Now for the applications into the quaternionic modular forms case, we need to know $\pi^{\mathfrak{o}^{\times}_{E}}$, where for each quadratic extension $E$.  Now since $\pi$ has trivial central character, for unramified case this means $\pi^{\mathfrak{o}^{\times}_{E}}=\pi^{E^{\times}}$, and for the ramified case this means that $\pi^{\mathfrak{o}_{E}^{\times}}=\pi^{E^{\times}}\bigoplus\pi^{\nu_{E}}$. 

 Let us recall our notations again. Let $\pi$ be a representation of $B^{\times}$ dihedrally induced from a ramified quadratic extension $K$ of $F$. Let $E$ be a ramified quadratic extension of $F$. 
\begin{corollary}
\begin{enumerate}
    \item $\dim \pi^{\mathfrak{o}^{\times}_{E}}=2$  if\begin{enumerate}
        \item 
    $E\cong K$, $q\equiv 3 \bmod4$ 
    \item  $E\cong L$, $q\equiv 1 \bmod4$ 
   \end{enumerate} 
\item $\dim \pi^{\mathfrak{o}^{\times}_{E}}=0$  if \begin{enumerate}
        \item 
    $E\cong K$, $q\equiv 1 \bmod4$ 
    \item  $E\cong L$, $q\equiv 3 \bmod4$ 
   \end{enumerate} 
 \end{enumerate}
\end{corollary}
    
\begin{remark}
    We need to restrict to the case for odd primes, since we use Lemma 3.1, and which is not true when $p=2$. This fails even in the basic case of $F=\mathbb{Q}_{2}$. It is well known that $x \in \mathcal{O}_{F}^{\times}$ is a square $\iff x\equiv 1 \bmod 8$. For example if $E=\mathbb{Q}_{2}(\sqrt{2})$. Clearly $-1=N_{E/F}(1+\sqrt{2})$, but $-1$ is not a square in $F$, since $-1\not\equiv1\bmod 8$ in $\mathcal{O}_{F}$. Hence our method of proof does not work for $p=2$.
\end{remark}
\begin{table}
    \centering
    \begin{tabular}{|c|c|c|c|}\hline
         Representation $\pi$&  Quadratic extension $E$&  $\dim \pi^{ \mathfrak{o}_{E}^{\times}}$ \\\hline
         $1$-dimensional& Any extension &  1 $\iff$ $\pi \circ N_{E/F}=1$ \\\hline
         Even conductor, minimal&  any extension& $e(E/F)$  \\\hline
                 Non-minimal representation, 
                 $\dim\pi>1$ & any quadratic extension & 0  \\ \hline

        Odd conductor, $\dim \pi>1$ & unramified extension & 1  \\\hline

        Odd conductor $\dim \pi>1$ & any ramified extension & ? \\ \hline
    \end{tabular}
    \caption{Summary of the needed results from \cite{MR4127854}}
    \label{table:Table}
\end{table} 
\section{\textbf{Global Results}}
Let $F$ be a totally real field of degree $d$ with $\mathfrak{o}=\mathfrak{o}_{F}$ its ring of integers. Let $\nu_{1}, \nu_{2},......\nu_{d}$ be the real embeddings of $F$. Let $B$ be a definite quaternion algebra over $F$, with $\mathfrak{D}_{B}$ its discriminant, and $\mathcal{O}$ be a special order in $B$ of level $\mathfrak{N}$. 
This means that for every $\mathfrak{p} | \mathfrak{D}$, $B_{\mathfrak{p}}$ is a division algebra and $\mathcal{O}_\mathfrak{p}$ be a Bass order of level $v_{\mathfrak{p}}(\mathfrak{N})$ in $B_{\mathfrak{p}}$. 
We know that every such local order can be written as $\mathcal{O}_{r}(E_{\mathfrak{p}})$, where $E_{\mathfrak{p}}$ is a quadratic extension of $F_{\mathfrak{p}}$. 
We say $\mathcal{O}_{\mathfrak{p}}$ is unramified if $E_{\mathfrak{p}}$ is unramified and ramified otherwise.
If $\mathcal{O}_{\mathfrak{p}}$ is unramified for some $\mathfrak{p}$, then it means $v_{\mathfrak{p}}(\mathfrak{N})$ is odd, while both parities are possible if the local order is ramified. 
We can write $\mathfrak{N}=\mathfrak{N}_{1}\mathfrak{N}_{2}\mathfrak{M}$, where $\mathfrak{N}_{1},\mathfrak{N}_{2}$ and $\mathfrak{M}$ are pairwise coprime, and 
\begin{enumerate}
    \item for each $\mathfrak{p}|\mathfrak{N}_{1}$, $\mathcal{O}_{\mathfrak{p}}$ is unramified,
    \item for each $\mathfrak{p}|\mathfrak{N}_{2}$, $\mathcal{O}_{\mathfrak{p}}$ is ramified,
    \item $\mathfrak{M}$ is coprime to $\mathfrak{D}$.
\end{enumerate} 
Hence this means for $\mathfrak{p}|\mathfrak{M}, B^{\times}_{\mathfrak{p}} \cong GL_{2}(F_{\mathfrak{p}})$. 
Let $M_{\mathbf{k}}(\mathcal{O})$ denote the space of quaternionic modular forms of weight $\mathbf{k}$, and level $\mathcal{O}$, and let $S_{\mathbf{k}}(\mathcal{O})$ be the corresponding space of cusp forms. 
These classical spaces corresponds to spaces of cuspidal automorphic representations of $B^{\times}(\mathbb{A})$. We say an eigenform is $\mathfrak{p}$-primitive, if the associated local representation $\pi_{\mathfrak{p}}$ is minimal. For an ideal $\mathfrak{a}\subset\mathfrak{o}$, we say an eigenform is $\mathfrak{a}$-primitive if it is $\mathfrak{p}$-primitive for all $\mathfrak{p}\mid\mathfrak{a}$.
The following propositions are in \cite{MR4127854}. 
\begin{proposition}\cite[(4.2),(4.3)]{MR4127854}
We have the following Hecke-module isomorphisms
\begin{eqnarray}
M_{\mathbf{k}}(\mathcal{O})&\cong& \bigoplus_{\pi} \pi_{\mathbf{k}}^{\mathcal{O^{\times}}} \textit{ and } \\
S_{\mathbf{k}}(\mathcal{O})&\cong& \bigoplus_{\dim \pi>1} \pi_{\mathbf{k}}^{\mathcal{O^{\times}}}, 
\end{eqnarray}
    where $\pi$ runs over irreducible automorphic representations of $B^{\times}(\mathbb{A})$, with trivial central character.
\end{proposition}
Let $\mathcal{O'}$ be a special order that contains $\mathcal{O}$. We call $\varphi \in M_{\mathbf{k}}(\mathcal{O}) $ an old form if $\varphi \in M_{\mathbf{k}}(\mathcal{O'})$. We denote the space generated by all such forms by $M_{\mathbf{k}}^{old}(\mathcal{O})$. Let $M_{\mathbf{k}}^{new}(\mathcal{O})$ be its orthogonal complement in $M_{\mathbf{k}}(\mathcal{O})$. One  analogously defines $S_{\mathbf{k}}^{old}(\mathcal{O})$ and $S_{\mathbf{k}}^{new}(\mathcal{O})$.
\begin{proposition}\cite[Proposition 4.4]{MR4127854}
    \begin{enumerate}
        \item We have the following decomposition 
        \begin{equation}
            S_{\mathbf{k}}^{new}(\mathcal{O})\cong \bigoplus \pi_{\mathbf{k}}^{\mathcal{O^{\times}}}, 
        \end{equation}
            where $\pi_{\mathbf{k}}$ runs over all irreducible automorphic representations of $B^{\times}(\mathbb{A})$ with trivial central character and dim$\pi_{\mathbf{k}}>1$, and $c(\pi_{\mathbf{k}})=\prod{\mathfrak{p}^{c(\pi_{\mathfrak{p}})}}=\mathfrak{N}$.
        \item As Hecke-modules we have the following isomorphism
        \begin{equation}
            S_{\mathbf{k}}^{old}(\mathcal{O})\cong \bigoplus S_{\mathbf{k}}^{new}(\mathcal{O'}),
        \end{equation}
        where $\mathcal{O'}$ runs over all special orders of B, properly containing $\mathcal{O}$.
    \end{enumerate}
\end{proposition}
By the Jacquet--Langlands transfer, we have a map $$JL:S_{\textbf{k}}(\mathcal{O}) \rightarrow S_{\textbf{k+2}}(\mathfrak{N}).$$ We want to study this map and its restriction  to $S_{\textbf{k}}^{new}(\mathcal{O})$. In \cite{MR4127854}, the author works with the case where $v_{\mathfrak{p}}(\mathfrak{N})$ is even whenever $\mathfrak{p}|\mathfrak{N}_{2}$. In the sequel, we aim to relax that condition. Let $\mathcal{O}_{\mathfrak{p}}=\mathcal{O}_{r}(K_{\mathfrak{p}})$. Let 
\[   
{E_{\mathfrak{p}}} \cong
     \begin{cases}
       K_{\mathfrak{p}} &\quad   \text{if q}  \equiv 3 \text{(mod 4)}, \\
    L_{\mathfrak{p}} &\quad   \text{if q}  \equiv 1 \text{(mod 4)}. \\ 
     \end{cases}
\]
Therefore from the previous section, whenever $dim(\pi_{\mathfrak{p}})>1$ we have
\[   
\dim \pi_{\mathfrak{p}}^{\mathcal{O}_{\mathfrak{p}}^{\times}} =
     \begin{cases}
       2 &\quad  \text{if } \pi_{\mathfrak{p}} \text{ induced from } {E_{\mathfrak{p}}} \\
    0 &\quad   \text{otherwise}.   \\ 
     \end{cases}
\]    
\begin{proposition} Let v$_{\mathfrak{p}}(\mathfrak{N}_{2})>1$, when $\mathfrak{p}|\mathfrak{N}_{2}$. 
    There exists a (non-canonical) homomorphism of $\mathcal{H}^{S}$ modules $$\textbf{JL}:S_{\mathbf{k}}^{new}(\mathcal{O}) \rightarrow S^{new}_{\mathbf{k+2}}(\mathfrak{N})$$
such that every $\mathfrak{D}$-primitive newform in $S^{new}_{\textbf{k+2}}(\mathfrak{N})$ satisfying
\begin{enumerate}
    \item the local representation $\pi_{\mathfrak{p}}$ is induced from $E_{\mathfrak{p}}$ if $v_{\mathfrak{p}}(\mathfrak{N}_{2})$ is odd, and
    \item the local representation $\pi_{\mathfrak{p}}$ is minimal if $v_{\mathfrak{p}}(\mathfrak{N}_{2})$ is even,
\end{enumerate} lies in the image. 
\begin{proof}
Let $\pi$ be a representation of $B^{\times}(\mathbb{A})$ which appears in the decomposition (15).
      By the global Jacquet--Langlands correspondence, $\pi$ corresponds to an irreducible automorphic representation $\pi'$ of $GL_{2}(\mathbb{A})$.
    For $\mathfrak{p}$ not dividing $\mathfrak{D}$, $B^{\times}_{\mathfrak{p}}\cong GL_{2}(F_{\mathfrak{p}}).$
When $\mathfrak{p}|\mathfrak{D}$, since $c(\pi_{\mathfrak{p}})>1$ and $\mathfrak{D}$-primitivity, $\pi'_{\mathfrak{p}}$ is supercuspidal. 
Such a representation occcurs if and only if $\pi_{\mathfrak{p}}^{\mathcal{O}_{\mathfrak{p}}^{\times}} \neq 0$ for $\mathfrak{p}|\mathfrak{D}$. Let $\mathfrak{p}|\mathfrak{n}_{2}$, with $v_{\mathfrak{p}}(\mathfrak{n}_{2})$ even. Combined with $\mathfrak{D}$-primitivity, this implies that $\pi_{\mathfrak{p}}$ is induced from $M_{\mathfrak{p}}$, the unramified extension of $F_{\mathfrak{p}}$. In this case, we have dim$(\pi_{\mathfrak{p}}^{\mathcal{O_{\mathfrak{p}}^{\times}}})=2$. 
Now let $\mathfrak{p}|\mathfrak{N}_{2}$ be such that $v_{\mathfrak{p}}(\mathfrak{N}_{2})$ is odd. This means that the corresponding local representation $\pi_{\mathfrak{p}}$ is induced from one of the ramified quadratic extension of $F_{\mathfrak{p}}$. We know that  $\pi_{\mathfrak{p}}^{\mathcal{O}_{\mathfrak{p}}^{\times}} \neq 0$ if and only if $\pi_{\mathfrak{p}}$ is induced from ${E_{\mathfrak{p}}}$.
Hence the result follows.
\end{proof}

\end{proposition}
\begin{remark}
    Let us say $v_{\mathfrak{p}_{2}}(\mathfrak{N}_{2})=1$. Let $\pi_{\mathbf{k}}=\bigotimes \pi_{\mathfrak{p}}$ occuring in RHS of (9). Clearly  $\pi_{{\mathfrak{p}}_{2}}$ a character. When it is not the trivial character, its image under the Jacquet--Langlands correspondence is a non-minimal representation. 
    
    \end{remark}
Next we will describe the image of the above map explicitly. Analogous to the case of classical modular forms written in the introduction, we can associate a Hecke-module isomorphism between the space of Hilbert cusp forms of a fixed level $\mathbf{k}$ and level $\mathfrak{N}$ as follows to a suitable space of automorphic representations as follows:
$$S^{{new}}_{\mathbf{k}}(\mathfrak{N})\simeq \bigoplus \pi_{\mathbf{k}}^{K_{1}(\mathfrak{N})},$$ where $\pi_{\mathbf{k}}$ runs of irreducible cuspidal automorphic representations with trivial central character and $K_{1}(\mathfrak{N})$ denote a suitable compact-open subgroup of level $\mathfrak{N}$. We refer the reader to \cite{MR4127854} for a detailed discussion on this dictionary.
Following \cite{MR4127854},
let us define the subspace $S^{[\mathfrak{a;b;c;d}]}_{\textbf{k+2}}(\mathfrak{N}) \subset S^{new}_{\textbf{k+2}}(\mathfrak{N})$, where $\mathfrak{a,b,c,d}$ are coprime ideals as follows
\[ S^{[\mathfrak{a; b; c; d}]}_{\mathbf{k}}(\mathfrak{N}) :\simeq \bigoplus \pi^{K_{1}(\mathfrak{N})}, \]
where $\pi$ runs over all irreducible weight $\mathbf{k}$ cuspidal automorphic representations of $GL(2,\mathbb{A})$, with trivial central character such that
\begin{enumerate}
    \item $c(\pi_{\mathfrak{p}})=v_{\mathfrak{p}}(\mathfrak{N})$ for $\mathfrak{p}|\mathfrak{abcd}$.
    \item $\pi_{\mathfrak{p}}$ is discrete series for $\mathfrak{p}|\mathfrak{a}$.
    \item $\pi_{\mathfrak{p}}$ is a minimal supercuspidal if $\mathfrak{p}|\mathfrak{b}$, which is induced from $E_{\mathfrak{p}}$ when $v_{\mathfrak{p}}(\mathfrak{b})$ is odd.
    \item $\pi_{\mathfrak{p}}$ is special if $\mathfrak{p|c}$.
\end{enumerate}
We write $S^{[\mathfrak{a; b;o; d}]}_{\mathbf{k}}(\mathfrak{N})$ as $S^{[\mathfrak{a; b; d}]}_{\mathbf{k}}(\mathfrak{N})$.
\begin{theorem}
    As $\mathcal{H}^{S}$-modules we have an isomorphism
    $$S_{\mathbf{k}}^{new}(\mathcal{O})\cong 2^{\#\{\mathfrak{p|N}_{2}\}} S^{[\mathfrak{N}_{1}; \mathfrak{N^{\prime}}_{2}; \mathfrak{N}^{\prime\prime}_{2};\mathfrak{M}]}_{\mathbf{k+2}}(\mathfrak{N}),$$ where $\mathfrak{N}^{\prime\prime}_{2}$ is such that $v_{\mathfrak{p}}(\mathfrak{N}^{\prime\prime}_{2})=2$ for every prime ideal $\mathfrak{p}$ dividing $\mathfrak{N}_{2}^{\prime\prime}$, and $\mathfrak{N}_{2}^{\prime}.\mathfrak{N}_{2}^{\prime\prime}=\mathfrak{N}_{2}$, and $\mathfrak{N}_{2}^{\prime}$ and $\mathfrak{N}_{2}^{\prime\prime}$ are coprime.
\end{theorem}
\begin{proof}
    Let $\pi$ be a representation appearing in LHS and let $\pi'$ be its Jacquet-Langlands transfer to $GL_{2}(\mathbb{A})$. Now we know if $\mathfrak{p|N_{1}}$, $v_{\mathfrak{p}}(\mathfrak{N}_{1})$ has to be odd. This means that $c(\pi_{\mathfrak{p}})$ is odd and hence by results above  $\pi_{\mathfrak{p}}^{\mathcal{O}_{\mathfrak{p}}^{\times}}$  is one-dimensional. Now if $v_{\mathfrak{p}}(\mathfrak{N_{2}})$ is even then either $\pi^{\prime}_{\mathfrak{p}}$ is non-minimal or it is a minimal supercuspidal. If $\pi_{\mathfrak{p}}$ is non-minimal, then  $\pi_{\mathfrak{p}}^{\mathcal{O}_{\mathfrak{p}}^{\times}} = 0$, unless $\dim (\pi_{\mathfrak{p}})=1$. So such representation contributes to the image of $JL$ only if $\pi_{\mathfrak{p}}^{\prime}$ is trivial or a  ramified twist of the Steinberg. If $\pi_{\mathfrak{p}}$ is minimal, then  $\pi_{\mathfrak{p}}^{\mathcal{O}_{\mathfrak{p}}^{\times}}$ is $2$-dimensional. Now  $\pi_{\mathfrak{p}}$ has an odd conductor, hence minimal. By above, we need to have $\pi_{\mathfrak{p}}$ to be induced from $E_{\mathfrak{p}}$, in order to have $\pi_{\mathfrak{p}}^{\mathcal{O^{\times}_{\mathfrak{p}}}}\neq0$, and these are $2$-dimensional, whenever non-zero. Hence the theorem follows.  
\end{proof}
In \cite{MR4127854}, the author gave a description of $S_{\textbf{k}}(\mathcal{O})$, when $\mathfrak{N}_{2}$ was cube-free. We aim to relax this assumption. Write $lev(\mathcal{O})=\mathfrak{N}_{1}\mathfrak{N}_{2}\mathfrak{M}$ as above
\begin{theorem}
    $$S_{\mathbf{k}}(\mathcal{O})\cong \bigoplus 2^{\{\mathfrak{p|{b}}\}}S_{\mathbf{k+2}}^{[{\mathfrak{a;b;c;M}}]}(\mathfrak{abcM})$$ where \begin{enumerate}
        \item $\mathfrak{D|abc}$,
        \item $\mathfrak{a|N}_{1}$ and $v_{\mathfrak{p}}(\mathfrak{N}_{1})$ is odd,
        \item $\mathfrak{bc|N}_{2}$, and
        \item for each $\mathfrak{p|c}$ and $v_{\mathfrak{p}}(\mathfrak{c})=1$ or $2$.
    \end{enumerate}
\end{theorem}
\begin{proof}
We know from $(17)$ that 
\begin{equation}
S_{\textbf{k}}(\mathcal{O})\cong \bigoplus S_{\textbf{k}}^{\mathfrak{ad-}new}(\mathcal{O^\prime (\mathfrak{adM}})),
\end{equation} where $\mathfrak{a}$ runs over ideals satisfying $(2)$, and $\mathfrak{d}$ runs over ideals such that $\mathfrak{D
|ad}$, and $\mathcal{O'(\mathfrak{adM}}))$ is a special order of level $\mathfrak{adM}$. Now fixing $\mathfrak{a}$ and $\mathfrak{d}$, we can write each space on RHS of $(19)$ as \begin{equation}\bigoplus S_{\textbf{k+2}}^{[\mathbf{\mathfrak{a;b;c;o}}]}(\mathfrak{abcM}),
\end{equation}
where $\mathfrak{b,c}$ runs over coprime divisors of $\mathfrak{d}$ such that $\mathfrak{bc=d}$, such that $\pi_{\mathfrak{p}}$ is higher dimensional for $\mathfrak{p|b}$ and one-dimensional for $\mathfrak{p|c}$. We know that whenever $\mathfrak{p|b}$, then $\dim \pi_{\mathfrak{p}}^{\mathcal{O^{\times}_{\mathfrak{p}}}}=2$. From Theorem 4.4, we get the result. 
\end{proof}
\printbibliography

@article {MR721997,
    AUTHOR = {Tunnell, Jerrold B.},
     TITLE = {Local {$\epsilon $}-factors and characters of {${\rm GL}(2)$}},
   JOURNAL = {Amer. J. Math.},
  FJOURNAL = {American Journal of Mathematics},
    VOLUME = {105},
      YEAR = {1983},
    NUMBER = {6},
     PAGES = {1277--1307},
      ISSN = {0002-9327,1080-6377},
   MRCLASS = {22E35 (11S37 20G05)},
  MRNUMBER = {721997},
MRREVIEWER = {Joe\ Repka},
       DOI = {10.2307/2374441},
       URL = {https://doi.org/10.2307/2374441},
}

@article {MR3056552,
    AUTHOR = {Pacetti, Ariel},
     TITLE = {On the change of root numbers under twisting and applications},
   JOURNAL = {Proc. Amer. Math. Soc.},
  FJOURNAL = {Proceedings of the American Mathematical Society},
    VOLUME = {141},
      YEAR = {2013},
    NUMBER = {8},
     PAGES = {2615--2628},
      ISSN = {0002-9939,1088-6826},
   MRCLASS = {11F70},
  MRNUMBER = {3056552},
MRREVIEWER = {Kimball\ L.\ Martin},
       DOI = {10.1090/S0002-9939-2013-11532-7},
       URL = {https://doi.org/10.1090/S0002-9939-2013-11532-7},
}

@article {MR4127854,
    AUTHOR = {Martin, Kimball},
     TITLE = {The basis problem revisited},
   JOURNAL = {Trans. Amer. Math. Soc.},
  FJOURNAL = {Transactions of the American Mathematical Society},
    VOLUME = {373},
      YEAR = {2020},
    NUMBER = {7},
     PAGES = {4523--4559},
      ISSN = {0002-9947,1088-6850},
   MRCLASS = {11F27 (11F41 11F70)},
  MRNUMBER = {4127854},
MRREVIEWER = {Hengfei\ Lu},
       DOI = {10.1090/tran/8077},
       URL = {https://doi.org/10.1090/tran/8077},
}

@article {MR268123,
    AUTHOR = {Atkin, A. O. L. and Lehner, J.},
     TITLE = {Hecke operators on {$\Gamma \sb{0}(m)$}},
   JOURNAL = {Math. Ann.},
  FJOURNAL = {Mathematische Annalen},
    VOLUME = {185},
      YEAR = {1970},
     PAGES = {134--160},
      ISSN = {0025-5831,1432-1807},
   MRCLASS = {10.20},
  MRNUMBER = {268123},
MRREVIEWER = {R.\ A.\ Rankin},
       DOI = {10.1007/BF01359701},
       URL = {https://doi.org/10.1007/BF01359701},
}

@article {MR337789,
    AUTHOR = {Casselman, William},
     TITLE = {On some results of {A}tkin and {L}ehner},
   JOURNAL = {Math. Ann.},
  FJOURNAL = {Mathematische Annalen},
    VOLUME = {201},
      YEAR = {1973},
     PAGES = {301--314},
      ISSN = {0025-5831,1432-1807},
   MRCLASS = {10D15 (22E50)},
  MRNUMBER = {337789},
MRREVIEWER = {T.\ Miyake},
       DOI = {10.1007/BF01428197},
       URL = {https://doi.org/10.1007/BF01428197},
}

@article {MR369263,
    AUTHOR = {Li, Wen Ch'ing Winnie},
     TITLE = {Newforms and functional equations},
   JOURNAL = {Math. Ann.},
  FJOURNAL = {Mathematische Annalen},
    VOLUME = {212},
      YEAR = {1975},
     PAGES = {285--315},
      ISSN = {0025-5831,1432-1807},
   MRCLASS = {10D05},
  MRNUMBER = {369263},
MRREVIEWER = {R.\ A.\ Rankin},
       DOI = {10.1007/BF01344466},
       URL = {https://doi.org/10.1007/BF01344466},
}

@article {MR1102225,
    AUTHOR = {Shemanske, Thomas R. and Walling, Lynne H.},
     TITLE = {Twists of {H}ilbert modular forms},
   JOURNAL = {Trans. Amer. Math. Soc.},
  FJOURNAL = {Transactions of the American Mathematical Society},
    VOLUME = {338},
      YEAR = {1993},
    NUMBER = {1},
     PAGES = {375--403},
      ISSN = {0002-9947,1088-6850},
   MRCLASS = {11F41},
  MRNUMBER = {1102225},
MRREVIEWER = {Jeffrey\ Stopple},
       DOI = {10.2307/2154461},
       URL = {https://doi.org/10.2307/2154461},
}

@incollection {MR3184337,
    AUTHOR = {Demb\'el\'e, Lassina and Voight, John},
     TITLE = {Explicit methods for {H}ilbert modular forms},
 BOOKTITLE = {Elliptic curves, {H}ilbert modular forms and {G}alois
              deformations},
    SERIES = {Adv. Courses Math. CRM Barcelona},
     PAGES = {135--198},
 PUBLISHER = {Birkh\"auser/Springer, Basel},
      YEAR = {2013},
      ISBN = {978-3-0348-0617-6},
   MRCLASS = {11F41 (11F60 11F67 11R52)},
  MRNUMBER = {3184337},
MRREVIEWER = {Gabor\ Wiese},
       DOI = {10.1007/978-3-0348-0618-3\_4},
       URL = {https://doi.org/10.1007/978-3-0348-0618-3_4},
}

@book {MR2234120,
    AUTHOR = {Bushnell, Colin J. and Henniart, Guy},
     TITLE = {The local {L}anglands conjecture for {$\rm GL(2)$}},
    SERIES = {Grundlehren der mathematischen Wissenschaften [Fundamental
              Principles of Mathematical Sciences]},
    VOLUME = {335},
 PUBLISHER = {Springer-Verlag, Berlin},
      YEAR = {2006},
     PAGES = {xii+347},
      ISBN = {978-3-540-31486-8},
   MRCLASS = {22E50 (11-02 11S37 22-02)},
  MRNUMBER = {2234120},
MRREVIEWER = {Alexandru\ Ioan\ Badulescu},
       DOI = {10.1007/3-540-31511-X},
       URL = {https://doi.org/10.1007/3-540-31511-X},
}

@article {MR977435,
    AUTHOR = {Hijikata, H. and Pizer, A. and Shemanske, T.},
     TITLE = {Orders in quaternion algebras},
   JOURNAL = {J. Reine Angew. Math.},
  FJOURNAL = {Journal f\"ur die Reine und Angewandte Mathematik. [Crelle's
              Journal]},
    VOLUME = {394},
      YEAR = {1989},
     PAGES = {59--106},
      ISSN = {0075-4102,1435-5345},
   MRCLASS = {11R52 (11F25 11F55 11S45 16A18)},
  MRNUMBER = {977435},
MRREVIEWER = {Shigeaki\ Tsuyumine},
       DOI = {10.1515/crll.1989.394.59},
       URL = {https://doi.org/10.1515/crll.1989.394.59},
}

@article{Brandt1943,
author = {Brandt, H.},
journal = {Jahresbericht der Deutschen Mathematiker-Vereinigung},
language = {ger},
pages = {23-57},
title = {Zur Zahlentheorie der Quaternionen.},
url = {http://eudml.org/doc/146327},
volume = {53},
year = {1943},
}

@article {MR1199206,
    AUTHOR = {Saito, Hiroshi},
     TITLE = {On {T}unnell's formula for characters of {${\rm GL}(2)$}},
   JOURNAL = {Compositio Math.},
  FJOURNAL = {Compositio Mathematica},
    VOLUME = {85},
      YEAR = {1993},
    NUMBER = {1},
     PAGES = {99--108},
      ISSN = {0010-437X,1570-5846},
   MRCLASS = {22E50 (11S37)},
  MRNUMBER = {1199206},
MRREVIEWER = {David\ Manderscheid},
       URL = {http://www.numdam.org/item?id=CM_1993__85_1_99_0},
}

@article {MR4492517,
    AUTHOR = {Martin, Kimball},
     TITLE = {Exact double averages of twisted {$L$}-values},
   JOURNAL = {Math. Z.},
  FJOURNAL = {Mathematische Zeitschrift},
    VOLUME = {302},
      YEAR = {2022},
    NUMBER = {3},
     PAGES = {1821--1854},
      ISSN = {0025-5874,1432-1823},
   MRCLASS = {11F67},
  MRNUMBER = {4492517},
MRREVIEWER = {Guohua\ Chen},
       DOI = {10.1007/s00209-022-03088-3},
       URL = {https://doi.org/10.1007/s00209-022-03088-3},
}

@article {MR960090,
    AUTHOR = {Hijikata, Hiroaki and Pizer, Arnold K. and Shemanske, Thomas
              R.},
     TITLE = {The basis problem for modular forms on {$\Gamma_0(N)$}},
   JOURNAL = {Mem. Amer. Math. Soc.},
  FJOURNAL = {Memoirs of the American Mathematical Society},
    VOLUME = {82},
      YEAR = {1989},
    NUMBER = {418},
     PAGES = {vi+159},
      ISSN = {0065-9266,1947-6221},
   MRCLASS = {11F11},
  MRNUMBER = {960090},
MRREVIEWER = {Minking\ Eie},
       DOI = {10.1090/memo/0418},
       URL = {https://doi.org/10.1090/memo/0418},
}

@article {MR1103429,
    AUTHOR = {Waldspurger, Jean-Loup},
     TITLE = {Correspondances de {S}himura et quaternions},
   JOURNAL = {Forum Math.},
  FJOURNAL = {Forum Mathematicum},
    VOLUME = {3},
      YEAR = {1991},
    NUMBER = {3},
     PAGES = {219--307},
      ISSN = {0933-7741,1435-5337},
   MRCLASS = {11F70 (11F30 11F32 11F37 22E50)},
  MRNUMBER = {1103429},
MRREVIEWER = {Stephen\ Gelbart},
       DOI = {10.1515/form.1991.3.219},
       URL = {https://doi.org/10.1515/form.1991.3.219},
}

@article {MR579066,
    AUTHOR = {Pizer, Arnold},
     TITLE = {An algorithm for computing modular forms on {$\Gamma
              \sb{0}(N)$}},
   JOURNAL = {J. Algebra},
  FJOURNAL = {Journal of Algebra},
    VOLUME = {64},
      YEAR = {1980},
    NUMBER = {2},
     PAGES = {340--390},
      ISSN = {0021-8693},
   MRCLASS = {10D12},
  MRNUMBER = {579066},
MRREVIEWER = {O.\ M.\ Fomenko},
       DOI = {10.1016/0021-8693(80)90151-9},
       URL = {https://doi.org/10.1016/0021-8693(80)90151-9},
}

@article {MR3665,
    AUTHOR = {Hecke, E.},
     TITLE = {Analytische {A}rithmetik der positiven quadratischen {F}ormen},
   JOURNAL = {Danske Vid. Selsk. Mat.-Fys. Medd.},
  FJOURNAL = {Det Kongelige Danske Videnskabernes Selskab Matematisk-Fysiske
              Meddelelser},
    VOLUME = {17},
      YEAR = {1940},
    NUMBER = {12},
     PAGES = {134},
      ISSN = {0023-3323},
   MRCLASS = {10.0X},
  MRNUMBER = {3665},
MRREVIEWER = {C.\ L.\ Siegel},
}

@article {MR80768,
    AUTHOR = {Eichler, Martin},
     TITLE = {\"Uber die {D}arstellbarkeit von {M}odulformen durch
              {T}hetareihen},
   JOURNAL = {J. Reine Angew. Math.},
  FJOURNAL = {Journal f\"ur die Reine und Angewandte Mathematik. [Crelle's
              Journal]},
    VOLUME = {195},
      YEAR = {1955},
     PAGES = {156--171 (1956)},
      ISSN = {0075-4102,1435-5345},
   MRCLASS = {33.0X},
  MRNUMBER = {80768},
MRREVIEWER = {O.\ F. G. Schilling},
       DOI = {10.1515/crll.1955.195.156},
       URL = {https://doi.org/10.1515/crll.1955.195.156},
}

@incollection {MR485698,
    AUTHOR = {Eichler, M.},
     TITLE = {The basis problem for modular forms and the traces of the
              {H}ecke operators},
 BOOKTITLE = {Modular functions of one variable, {I} ({P}roc. {I}nternat.
              {S}ummer {S}chool, {U}niv. {A}ntwerp, {A}ntwerp, 1972)},
    SERIES = {Lecture Notes in Math.},
    VOLUME = {Vol. 320},
     PAGES = {75--151},
 PUBLISHER = {Springer, Berlin-New York},
      YEAR = {1973},
   MRCLASS = {10D10},
  MRNUMBER = {485698},
}

@book {MR4279905,
    AUTHOR = {Voight, John},
     TITLE = {Quaternion algebras},
    SERIES = {Graduate Texts in Mathematics},
    VOLUME = {288},
 PUBLISHER = {Springer, Cham},
      YEAR = {2021},
     PAGES = {xxiii+885},
      ISBN = {978-3-030-56692-0},
   MRCLASS = {11R52 (11-02 11E12 11F06 16H05 16U60 20H10)},
  MRNUMBER = {4279905},
MRREVIEWER = {Juliusz\ Brzezi\'nski},
       DOI = {10.1007/978-3-030-56694-4},
       URL = {https://doi.org/10.1007/978-3-030-56694-4},
}

@unpublished{KM,
  author = "{Knightly, Andrew and Martin, Kimball}",
  title  = "Counting newforms with given supercuspidal components",
  note   = "Manuscript in preparation",
  year   = {2025},
  shorthand = {KM25}
    
}

@article {MR4269428,
    AUTHOR = {Dieulefait, Luis Victor and Pacetti, Ariel and Tsaknias,
              Panagiotis},
     TITLE = {On the number of {G}alois orbits of newforms},
   JOURNAL = {J. Eur. Math. Soc. (JEMS)},
  FJOURNAL = {Journal of the European Mathematical Society (JEMS)},
    VOLUME = {23},
      YEAR = {2021},
    NUMBER = {8},
     PAGES = {2833--2860},
      ISSN = {1435-9855,1435-9863},
   MRCLASS = {11F03 (11F11)},
  MRNUMBER = {4269428},
MRREVIEWER = {Spencer\ Hamblen},
       DOI = {10.4171/jems/1073},
       URL = {https://doi.org/10.4171/jems/1073},
}

@article {MR3762695,
    AUTHOR = {Martin, Kimball},
     TITLE = {The {J}acquet-{L}anglands correspondence, {E}isenstein
              congruences, and integral {$L$}-values in weight 2},
   JOURNAL = {Math. Res. Lett.},
  FJOURNAL = {Mathematical Research Letters},
    VOLUME = {24},
      YEAR = {2017},
    NUMBER = {6},
     PAGES = {1775--1795},
      ISSN = {1073-2780,1945-001X},
   MRCLASS = {11R39 (11F66)},
  MRNUMBER = {3762695},
MRREVIEWER = {Jack\ Shotton},
       DOI = {10.4310/MRL.2017.v24.n6.a11},
       URL = {https://doi.org/10.4310/MRL.2017.v24.n6.a11},
}
\end{document}